\documentclass[letterpaper, 10 pt, conference]{ieeeconf}  % Comment this line out if you need a4paper
\IEEEoverridecommandlockouts                              % This command is only needed if 
\let\labelindent\relax

\usepackage[utf8]{inputenc}
\usepackage{graphics} % for pdf, bitmapped graphics files
\usepackage{amsmath} % assumes amsmath package installed
\usepackage{amssymb}  % assumes amsmath package installed
\usepackage{amsfonts}
\usepackage{comment}
\usepackage{enumitem}
\usepackage{cases,xspace,enumitem}
\usepackage{stmaryrd}
 \usepackage{enumitem}

\usepackage{version}

\includeversion{arxiv}
\excludeversion{lcss}

\usepackage{tikz}
\usetikzlibrary{positioning}

\usepackage[prependcaption,colorinlistoftodos]{todonotes}

\newcommand{\todoing}[1]{\todo[inline,color=green!20, linecolor=orange!250]{\small#1}}

\newcommand{\sashatodo}[1]{\todoing{\textbf{From Sasha:} #1}}
\definecolor{gnred}{RGB}{255,91,89}

\definecolor{gnred1}{RGB}{71,0,0} % 470000
\definecolor{gnred2}{RGB}{117,0,0} % 750000
\definecolor{gnred3}{RGB}{164,0,0} % a40000
\definecolor{gnred4}{RGB}{211,0,0} % d30000
\definecolor{gnred5}{RGB}{255,0,0} % FF0000
\definecolor{gnred6}{RGB}{255,42,34} % FF2a22
\definecolor{gnred7}{RGB}{255,91,89} % ff5b59 --- favorite

\definecolor{gnblue1}{RGB}{0,36,71}   % 002447  darker
\definecolor{gnblue2}{RGB}{0,60,118}  % 003c76
\definecolor{gnblue3}{RGB}{0,85,164}  % 0055A4
\definecolor{gnblue4}{RGB}{0,108,212} % 006CD4
\definecolor{gnblue5}{RGB}{0,133,255}  % 0085ff
\definecolor{gnblue6}{RGB}{35,156,255} % 239cff
\definecolor{gnblue7}{RGB}{88,177,255} % 58b1ff

\definecolor{gnbrown1}{RGB}{71,27,0}  % 471b00
\definecolor{gnbrown2}{RGB}{117,45,0} % 752d00
\definecolor{gnbrown3}{RGB}{164,62,0} % a43e00
\definecolor{gnbrown4}{RGB}{211,80,0} % d35000
\definecolor{gnbrown5}{RGB}{255,97,0} % ff6100
\definecolor{gnbrown6}{RGB}{255,127,26} % ff7f1a
\definecolor{gnbrown7}{RGB}{255,155,86} % ff9b56

\newcounter{saveenum}
\renewcommand{\theenumi}{(\roman{enumi})}

\newcommand\Item[1][]{%
  \ifx\relax#1\relax  \item \else \item[#1] \fi
  \abovedisplayskip=0pt\abovedisplayshortskip=0pt~\vspace*{-\baselineskip}}

\usepackage{hyperref}
\hypersetup{
    colorlinks=true,
    linkcolor=blue,
    filecolor=magenta,      
    urlcolor=cyan,
    pdfpagemode=FullScreen,
    }

\usepackage{amsthm}

\newtheoremstyle{ieeeconf}
  {0pt}   % ABOVESPACE
  {0pt}   % BELOWSPACE
  {\normalfont}  % BODYFONT
  {\parindent}       % INDENT (empty value is the same as 0pt)
  {\itshape} % HEADFONT
  {:}         % HEADPUNCT
  { } % HEADSPACE
  {\thmname{#1} \thmnumber{#2}\thmnote{ (#3)}} % CUSTOM-HEAD-SPEC
\makeatletter
\renewenvironment{proof}[1][\proofname]{\par
  \pushQED{\qed}%
  \normalfont \topsep\z@
  \trivlist
  \item[\hskip2em
        \itshape
    #1\@addpunct{:}]\ignorespaces
}{%
  \popQED\endtrivlist\@endpefalse
}
\makeatletter

\theoremstyle{ieeeconf}

\newtheorem{theorem}{Theorem}
\newtheorem{lemma}[theorem]{Lemma}
\newtheorem{prop}[theorem]{Proposition}
\newtheorem{corollary}[theorem]{Corollary}

\newtheorem{definition}[theorem]{Definition}

\newtheorem{remark} [theorem]{Remark}

\newtheorem{example}[theorem]{Example}

\usepackage{hyperref}%
\hypersetup{colorlinks=true, linkcolor=blue, breaklinks=true, urlcolor=blue}

\newcommand{\until}[1]{\{1,\dots, #1\}}
\newcommand{\subscr}[2]{#1_{\textup{#2}}}

\newcommand{\setdef}[2]{\{#1 \; | \; #2\}}

\newcommand{\map}[3]{#1: #2 \rightarrow #3}

\newcommand{\real}{\ensuremath{\mathbb{R}}}
\newcommand{\realpositive}{\ensuremath{\mathbb{R}}_{>0}}
\newcommand{\realnonnegative}{\ensuremath{\mathbb{R}}_{\ge 0}}

\newcommand\oprocendsymbol{\hbox{$\square$}}
\newcommand\oprocend{\relax\ifmmode\else\unskip\hfill\fi\oprocendsymbol}

\renewcommand{\theenumi}{(\roman{enumi})}

\DeclareSymbolFont{bbold}{U}{bbold}{m}{n}
\DeclareSymbolFontAlphabet{\mathbbold}{bbold}
\newcommand{\vect}[1]{\mathbbold{#1}}
\newcommand{\vectorones}[1][]{\vect{1}_{#1}}
\newcommand{\vectorzeros}[1][]{\vect{0}_{#1}}

\newcommand{\e}{\textrm{e}}

\newcommand*{\mydoi}[1]{\href{http://dx.doi.org/#1}{\includegraphics[width=.75em]{doi.png}}}

\newcommand{\jac}[1]{D\mkern-0.75mu{#1}}

\newcommand{\seminorm}[1]{{\left\vert\kern-0.25ex\left\vert\kern-0.25ex\left\vert #1
		\right\vert\kern-0.25ex\right\vert\kern-0.25ex\right\vert}}

\newcommand{\semimeasure}[1]{\mu_{\seminorm{\cdot}}\kern-0.5ex\left(#1\right)}

\DeclareMathOperator*{\argmin}{arg\,min}
\DeclareMathOperator{\proj}{Proj}

\DeclareMathOperator{\dom}{Dom}
\DeclareMathOperator{\interior}{int}

\newcommand{\mcX}{\mathcal{X}}
\newcommand{\mcU}{\mathcal{U}}
\newcommand{\mcY}{\mathcal{Y}}

\newcommand{\mcC}{\mathcal{C}}
\newcommand{\mcK}{\mathcal{K}}
\newcommand{\mcO}{\mathcal{O}}

\newcommand{\sat}{\mathrm{sat}}

\title{Exponential Stability of Parametric Optimization-Based Controllers \\ via Lur'e Contractivity}

\author{Alexander Davydov and Francesco Bullo\thanks{This work was in part supported by AFOSR project FA9550-21-1-0203 and NSF Graduate Research Fellowship under Grant 2139319. The first author thanks Pol Mestres and Anton Proskurnikov for insightful conversations.}%
\thanks{Alexander Davydov, and Francesco Bullo are with the Center for Control, Dynamical 
Systems, and Computation, UC Santa Barbara, Santa Barbara, CA 93106 USA. {\tt\small 
davydov@ucsb.edu, bullo@ucsb.edu}.}}

\begin{document}
\maketitle
\thispagestyle{empty}
\pagestyle{empty}

\begin{abstract}
  In this letter, we investigate sufficient conditions for the exponential
  stability of LTI systems driven by controllers derived from parametric
  optimization problems. Our primary focus is on parametric projection
  controllers, namely parametric programs whose objective function is the
  squared distance to a nominal controller. Leveraging the virtual system
  method of analysis and a novel contractivity result for Lur'e systems, we
  establish a sufficient LMI condition for the exponential stability of an
  LTI system with a parametric projection-based controller. Separately, we
  prove additional results for single-integrator systems. Finally, we apply
  our results to state-dependent saturated control systems and control
  barrier function-based control and provide numerical simulations.
\end{abstract}

\section{Introduction}
Controllers derived from optimization problems are ubiquitous in systems
and control.
%%
%% One large class of optimization-based controllers adopt the following
%% procedure: (i) solve an optimal control problem offline, such as LQR, LQG,
%% or Hamilton-Jacobi PDE and (ii) use the resulting controller for the plant
%% to close the loop.
%%
One large class of optimization-based controllers are based upon (i)
solving an optimal control problem offline, such as LQR, LQG, or
Hamilton-Jacobi PDE and (ii) closing the loop with the resulting
controller.
Recent interest has focused on a different class of optimization-based
controllers, that solve optimization problems at every time-step of the
dynamic evolution of the plant. Namely, such controllers are solutions to
\emph{parametric optimization problems}, i.e., programs that are functions
of the state of the system. Examples of these controllers include
model-predictive control~\cite{FB-AB-MM:17}, online feedback
optimization~\cite{GB-JC-JIP-EDA:22}, and control barrier (or Lyapunov)
function-based control~\cite{ADA-XX-JWG-PT:17}.
While stability and robustness properties of the first class of
optimization-based controllers are well understood, fewer studies have
focused on stability and robustness properties of parametric
optimization-based controllers.

%% In this letter, we study sufficient conditions for exponential stability
%% for LTI systems with a special class of parametric optimization-based
%% controllers. Specifically, given a nominal feedback controller, $k(x) =
%% Kx$, we consider projection-based controllers which minimally modify $k(x)$
%% so that it satisfies constraints on the input, e.g., forward-invariance of
%% a given set in the context of control barrier functions (CBFs).

\textit{Literature Review:} Parametric optimization is a rich subdiscipline of optimization which studies solutions of optimization problems as a function of a parameter; see the textbook~\cite{BB-JG-DK-BK-KT:82}. Parametric optimization is ubiquitous in systems and control, especially in model predictive control~\cite{FB-AB-MM:17} and CBF-based control~\cite{ADA-XX-JWG-PT:17}. Closed-form solutions for certain classes of parametric programs were studied in~\cite[Chapter~5]{FB-AB-MM:17}. However, closed-form solutions are not always attainable. Regularity of solutions to parametric programs, namely establishing smoothness properties of their solutions, is a classical problem and has even pervaded systems and control~\cite{BJM-MJP-ADA:15,PM-AA-JC:23}. Compared to regularity results, there are fewer results on the stability of control systems with parametric optimization-based controllers. 

One class of systems for which there have been results on stability and safety of systems driven by parametric optimization-based controllers are those coming from CBFs and control Lyapunov functions (CLFs)~\cite{ADA-XX-JWG-PT:17}. In these works, CLF and CBF constraints are jointly enforced in a state-dependent quadratic program (QP). To guarantee feasibility of the QP when the CLF and CBF inequalities cannot be jointly satisfied, the stability is commonly relaxed by introducing a slack variable. This relaxation results in a lack of stability guarantee even for arbitrarily large penalties on the slack variable~\cite{MJ:18}. Recent work,~\cite{PM-JC:23}, studied a variant of the CLF-CBF QP controller and demonstrates how to tune the penalty parameter and how to estimate the basin of attraction of the origin.

%\fbtodo{in Figure 2 (center panel), explain where does $\e^{-t}$ come
%  from. i thought you would include the upper bound from Theorem 1}

%\fbtodo{possibly improve colors of curves in Figure 2}

\textit{Contributions:} We consider LTI
systems equipped with parametric projection-based controllers. 
As our main contribution, assuming
linearity of the nominal controller and various well-posedness conditions,
we obtain LMI-based sufficient conditions for exponential stability and the
existence of a global Lyapunov function.

Our analysis is based upon the virtual system methods in contraction
theory and contractivity of Lur'e systems. For context, contraction theory 
is a computationally-friendly
notion of robust nonlinear stability~\cite{FB:23-CTDS} and the virtual
system method, first proposed in~\cite{WW-JJES:05}, is an analysis approach
to establish exponential convergence for systems satisfying certain weak
contractivity properties. As a tutorial contribution, we provide a novel
review of the virtual system method in Section~\ref{sec:virtual-sys}.
Specifically, we show that LTI systems with parametric
projection-based controllers are in Lur'e form with state-dependent
nonlinearity and that an appropriate virtual system can be designed in
standard Lur'e form.

As our second main contribution, we establish in Theorem~\ref{prop:Lure1} a
novel necessary and sufficient condition for contractivity of Lur'e systems
with cocoercive nonlinearities. In contrast, in~\cite{VA-ST:19} and~\cite[Proposition~4]{MG-VA-ST-DA:23},
monotone and Lipschitz nonlinearities are considered yet only sufficient conditions are provided.
In other words, by studying a smaller class of nonlinearities, we are able to find the sharpest condition for contractivity.
%whereby we demonstrate that their condition is a special case of
%ours. 
Additionally, our result is stronger
than~\cite[Theorem~4.2]{LDA-MC:13} since our condition is necessary and
sufficient.

%Next, leveraging an LMI-based contractivity result for Lur'e systems
%from~\cite{VA-ST:19}, we establish contractivity of the virtual system and
%therefore global exponential stability for the original system.

As our third main contribution, we study the special LTI case of single
integrators. We establish that all trajectories of the closed-loop system
converge to the set of equilibria and that all trajectories converging to
the origin do so exponentially fast with a known rate.
While there are related results in the CBF/CLF
literature~\cite{MFR-APA-PT:21,XT-DVD:24}, this convergence result for the
general class of parametric projection-based controllers is novel, to the
best of our knowledge.

Finally, we study two applications, namely state-dependent saturated
control systems and CBF-based control. For state-dependent saturated
control systems, the maximal control effort depends on the state of the
system and we demonstrate that our sufficient condition can be readily
applied to yield a condition for global exponential stability. In CBF-based
control, we consider a single integrator avoiding an obstacle and
demonstrate that the results hold and provide evidence that the estimated
exponential rate of convergence is tight. Specifically, we numerically observe
that, in the case of single integrator dynamics, one does not need to
enforce any CLF decrease condition to guarantee stability to the origin.

%\begin{arxiv}
%	Since this document is an extended version, it includes
%\end{arxiv}
%Motivated by control barrier functions (CBF), we investigate sufficient conditions for the exponential stability of LTI systems in feedback with an parametric optimization-based controller. The key tools in the analysis are the method of virtual systems (referred to as partial contraction in~\cite{WW-JJES:05})~\cite{FB:23-CTDS} and sufficient conditions for contractivity or exponential stability of Lur'e systems~\cite{MG-VA-ST-DA:23}. CBFs have been used in the design of nonlinear programming solvers~\cite{AA-JC:23}.

\section{Prerequisite Material}
\subsection{Contraction Theory}
Given a symmetric positive-definite matrix $P \in \real^{n\times n}$, we let $\|\cdot\|_{P}$ be $P$-weighted $\ell_2$ norm $\|x\|_{P} := \sqrt{x^\top P x}$, $x \in \real^n$ and write $\|\cdot\|_2$ if $P = I_n$. 
Given two normed spaces $\mcX$, $\mcY$ a map $\map{F}{\mcX}{\mcY}$ is Lipschitz with constant $\ell \geq 0$ if for all $x_1,x_2 \in \mcX$, it holds that
$
\|F(x_1)-F(x_2)\|_{\mcY} \leq \ell\|x_1-x_2\|_{\mcX}.
$

Let the vector field $\map{F}{\real_{\geq 0} \times \real^n}{\real^n}$ be continuous in its first argument and Lipschitz in its second, with $(t,x) \mapsto F(t,x)$. Further let $P \in \real^{n \times n}$ be symmetric and positive definite, and let there exist a constant $c >0$ referred as \emph{contraction rate}. We say that $F$ is \emph{strongly infinitesimally contracting with respect to $\|\cdot\|_P$ with rate $c$} if for all $x_1,x_2 \in \real^n$ and $t \geq 0$, 
\begin{equation}\label{eq:contraction-ineq}
	(F(t,x_1)-F(t,x_2))^\top P(x_1-x_2) \leq -c\|x_1-x_2\|_P^2.
\end{equation} 
%Note that this Jacobian exists for almost every $x$ in view of Rademacher's theorem. Typically this definition is given for maps which are continuously differentiable in $x$, but it is shown in~\cite[Theorem~16]{AD-AVP-FB:22q-arxiv}, that it applies just as well for Lipschitz ones.
If $x(\cdot)$ and $y(\cdot)$ are two trajectories satisfying $\dot{x}(t) = F(t,x(t)), \dot{y}(t) = F(t,y(t))$, then $\|x(t) - y(t)\|_P \leq \e^{-c(t-t_0)}\|x(t_0) - y(t_0)\|_P$ for all $t \geq t_0 \geq 0$. 

%One of the main benefits of contraction theory is that, with just a single condition, it ensures global
%exponential convergence, along with other useful robustness properties. 
We refer to~\cite{FB:23-CTDS} for a recent review of these tools.

\newcommand{\odeflowtx}[2]{\phi_{#1}(#2)}
\newcommand{\fvirtual}{\subscr{f}{virtual}}
\subsection{Virtual System Method for Convergence Analysis}\label{sec:virtual-sys}
The \emph{virtual system} analysis approach is a method to study the
asymptotic behavior of a dynamical system that may not enjoy contracting
properties. The virtual system approach was first proposed in~\cite{WW-JJES:05}, but we follow the systematic procedure advocated for in~\cite[Section~5.7]{FB:23-CTDS}. For completeness sake, we describe this procedure below.

The virtual system analysis approach is as follows. We are given
a dynamical system
\begin{equation}
	\dot{x}=f(x), \quad x(0)=x_0\in\real^n
\end{equation}
and we let $\odeflowtx{x_0}{t}$ denote a solution from initial
condition $x(0)=x_0$. The analysis proceeds in three steps:
\begin{enumerate}
	\item \emph{design} a time-varying dynamical system, called the
	\emph{virtual system}, of the form
	\begin{equation}
		\dot{y}=\fvirtual(y,\odeflowtx{x_0}{t}), \quad y\in\real^d
	\end{equation}
	satisfying a strong infinitesimal contractivity property with respect to
	an appropriate norm, e.g., the existence of a positive definite matrix $P \in \real^{d \times d}$ and a scalar $c > 0$ such that for all $y_1,y_2 \in \real^d, z \in \real^{n}$:
	\begin{equation*}
		%P\frac{\partial \fvirtual}{\partial y}(y,z) + \frac{\partial \fvirtual}{\partial y}(y,z)^\top P \preceq -2cP;
		\hspace{-2em}(\fvirtual(y_1,z)-\fvirtual(y_2,z))^\top P(y_1-y_2) \leq -c\|y_1-y_2\|_P^2;
	\end{equation*}
	(The vector field is called virtual since it is
	different from the nominal vector field, $f$, and does not
	correspond to any physically meaningful variation of $f$.)
	
	\item\label{vs:step:2} \emph{select} two specific solutions of the virtual
	system and state their incremental stability property:
	\begin{equation}
		\|y_1(t) - y_2(t)\|_P \leq \e^{-ct}\|y_1(0)-y_2(0)\|_P;
	\end{equation}
	\item \emph{infer} properties of the trajectory, $\odeflowtx{x_0}{t}$, of the
	nominal system.
\end{enumerate}

For example, if $d=n$ and $f(x)=\fvirtual(x,x)$, then one can see that
$\odeflowtx{x_0}{t}$ is a solution for both systems and is often selected
as one of the two specific solutions in step~\ref{vs:step:2}. Additionally,
if $\fvirtual(\vectorzeros[n],z) = \vectorzeros[n]$ for all $z \in
\real^n$, then $\vectorzeros[n]$ is an equilibrium point for the virtual
system and can be selected as one of the specific solutions.

\section{Absolute Contractivity of Lur'e Systems}
Consider the Lur'e system
\begin{equation}\label{eq:Lure}
	\dot{x} = Ax + B\varphi(Kx,t),
\end{equation}
where $\map{\varphi}{\real^m \times \realnonnegative}{\real^m}$ is cocoercive in its first argument and continuous in its second argument. Specifically, there exists a constant $\rho > 0$ such that for all $y_1,y_2 \in \real^m, t \geq 0$, 
\begin{equation}\label{eq:monotonicity}
	(\varphi(y_1,t)-\varphi(y_2,t))^\top(y_1-y_2) \geq \rho\|\varphi(y_1,t)-\varphi(y_2,t)\|_2^2.
\end{equation}
Notably, cocoercivity,~\eqref{eq:monotonicity}, implies that $\varphi$ is monotone and Lipschitz continuous with constant $\rho^{-1}$ in its first entry. Many standard nonlinearities satisfy cocoercivity including projections onto convex sets and nonlinearities of the form $\varphi(x) = (\varphi_1(x_1),\dots,\varphi_m(x_m))$ where each $\varphi_i$ is slope-restricted between $0$ and $\rho^{-1}$.

Akin to the classical problem of absolute stability, \emph{absolute contractivity} is the property that the system~\eqref{eq:Lure} is strongly infinitesimally contracting for any nonlinearity $\varphi$ obeying the constraint~\eqref{eq:monotonicity}. 

\begin{theorem}[Necessary and sufficient condition for absolute contractivity]\label{prop:Lure1}
	Consider the Lur'e system~\eqref{eq:Lure} and let $P \in \real^{n \times n}$ be positive definite. The system~\eqref{eq:Lure} is strongly infinitesimally contracting with respect to $\|\cdot\|_{P}$ with rate $\eta > 0$ for any $\varphi$ satisfying~\eqref{eq:monotonicity} if and only if there exists $\lambda \geq 0$ such that
	\begin{equation}\label{eq:LMI-general}
		\begin{bmatrix}
			A^\top P + PA + 2\eta P & PB + \lambda K^\top \\ B^\top P + \lambda K & -2\lambda \rho I_m
		\end{bmatrix} \preceq 0.
	\end{equation}
\end{theorem}
\begin{proof}
  Employing the shorthand $\Delta x = x_1-x_2, \Delta y = y_1-y_2 = K\Delta
  x, \Delta u_t = \varphi(y_1,t)-\varphi(y_2,t)$, the contractivity
  condition~\eqref{eq:contraction-ineq} for the system~\eqref{eq:Lure} is
  equivalently rewritten as
  \begin{equation}\label{eq:quad-ineq-contraction}
    \Delta x^\top (PA + A^\top P + 2\eta P)\Delta x + 2\Delta x^\top PB \Delta u_t \leq 0.
  \end{equation}
  Moreover, the cocoercivity
  condition~\eqref{eq:monotonicity} is equivalent to
  \begin{align}\label{eq:quad-ineq-monotone}
    & \Delta u_t^\top(\rho\Delta u_t - K\Delta x) \leq 0.
  \end{align}
 % To see this fact, note that the
 % inequality~\eqref{eq:quad-ineq-monotone-2} is trivially true when $\Delta
 % u_t = \vectorzeros[m]$, so we assume $\Delta u_t \neq
 % \vectorzeros[m]$. Inequality~\eqref{eq:quad-ineq-monotone-2} directly
 % implies the left inequality of~\eqref{eq:monotonicity}.  Additionally,
 % multiplying both sides of~\eqref{eq:quad-ineq-monotone-2} by $\rho\Delta
 % u_t^\top \Delta y /\|\Delta u_t\|_2^2>0$ implies the right inequality
 % of~\eqref{eq:monotonicity} by Cauchy-Schwarz. The reverse implication
 % holds by multiplying the inequality~\eqref{eq:monotonicity} by the
 % positive quantity $\rho^{-1}\Delta u_t^\top \Delta y/\|\Delta y\|_2^2$
 % (noting that the result is trivial when $\Delta y = \vectorzeros[m]$
 % since $\Delta y = \vectorzeros[m] \implies \Delta u_t =
 % \vectorzeros[m]$).  
 Asking when the
  inequality~\eqref{eq:quad-ineq-monotone}
  implies~\eqref{eq:quad-ineq-contraction} is equivalent to the
  inequality~\eqref{eq:LMI-general} in light of the necessity and
  sufficiency of the S-procedure~\cite{IP-TT:07}.
\end{proof}
Note that the condition in~\cite[Theorem~2]{VA-ST:19} corresponds to the inequality~\eqref{eq:LMI-general} with $\lambda = 1$. Moreover the matrix in~\eqref{eq:LMI-general} has $A^\top P + PA + 2\eta P$ in its $(1,1)$ block compared to $A^\top P + PA + \eta I_n$ in~\cite{VA-ST:19}. This modification ensures that the inequality~\eqref{eq:contraction-ineq} holds rather than a related inequality with $-\frac{\eta}{2} \|x_1-x_2\|_2^2$ on the right-hand side.
On the other hand, there exist mappings $\varphi$ that are monotone and Lipschitz but not cocoercive. Thus, we consider a smaller class of mappings in Theorem~\ref{prop:Lure1}, but are able to provide necessary and sufficient conditions for absolute contractivity.
%Moreover, while~\eqref{eq:LMI-general} is not linear in $\eta$ and $P$, at fixed $\eta$, the inequality is an LMI and thus to maximize $\eta$, one could employ a bisection algorithm.
%\begin{prop}[{\cite[Lemma~2]{VC-AG-AD-GR-FB:23c}}]\label{prop:sylvester}
%	\label{prop:prod_sym_matrices}
%	Let $A_1 = SQ \in \real^{n\times n}$ and $A_2 = QS \in \real^{n\times n}$ where $S$, $Q$ are symmetric, with $Q \succ 0$. Then, for each $i\in \{\,1,2\,\}$, the spectrum of $A_i$ has only real values and has the same number of negative,
%	zero, and positive eigenvalues as $S$. Moreover,
%	$\mu_{Q}(A_i) = \max\setdef{\lambda}{\lambda \text{ is an eigenvalue of } A_i}$.
%\end{prop}

\section{Parametric Projection-Based Controllers}
We are interested in studying a continuous-time LTI system being driven by an parametric optimization-based controller. We say that the optimization problem is parametric since it is a function of the state.  Specifically, we look at parametric projection-based controllers. More concretely, for $A \in \real^{n \times n}, B \in \real^{n \times m}, \map{u^\star}{\real^n}{\real^m}, \map{k}{\real^n}{\real^m}, \map{g}{\real^n \times \real^m}{\real^p}$, the LTI system and controller are given by:
\begin{equation}\label{eq:LTI}
	\begin{aligned}
		\dot{x} &= Ax + Bu^\star(x), \\
		u^{\star}(x) := &\argmin_{u \in \real^m} && \hspace{-3.5em}\frac{1}{2}\|u - k(x)\|_2^2 \\
		& \text{s.t.}  &&\hspace{-3.5em} g(x,u) \leq \vectorzeros[p].
	\end{aligned}
\end{equation}
In the context of the parametric optimization problem in~\eqref{eq:LTI}, $k$ denotes a nominal feedback controller and $g$ captures constraints on the controller as a function of the state. 
Such controllers commonly arise in CLF and CBF theory, where the parametric optimization problem in~\eqref{eq:LTI} is used to enforce that $u^\star$ either causes the closed-loop system to decrease a specified Lyapunov function or keep a certain set forward-invariant, respectively~\cite{ADA-XX-JWG-PT:17}. 

The question we aim to answer in this letter is the following: \textbf{What are conditions on the LTI system and the parametric optimization problem to ensure exponential stability of~\eqref{eq:LTI}?} Our main method for establishing sufficient conditions for exponential stability will be via the virtual system method in Section~\ref{sec:virtual-sys}.

\newcommand{\taumax}[1]{\subscr{\tau}{max}(#1)}
\subsection{Well-Posedness and Regularity of Solutions}
In order to study the dynamical system~\eqref{eq:LTI}, we need to ensure that it is well-posed. In other words, we need to ensure that the controller $u^\star$ is suitably regular, i.e., at least continuous. Several works in the literature have studied sufficient conditions for certain regularity of $u^\star$, e.g., continuity, Lipschitzness, or differentiability~\cite{AVF:76,BJM-MJP-ADA:15,PM-AA-JC:23}. In this work, we utilize the following proposition from~\cite{PM-AA-JC:23} which provides a sufficient condition for $u^\star$ to be continuous.
\begin{prop}[{\cite[Proposition~4]{PM-AA-JC:23}}]\label{prop:regularity}
	Consider the map $\map{u^\star}{\real^n}{\real^m}$ defined via the solution to the parametric optimization problem
	\begin{equation}\label{eq:PNLP}
		\begin{aligned}
			u^{\star}(x) := &\argmin_{u \in \real^m} && f(x,u) \\
			 &\text{s.t.}  && g(x,u) \leq \vectorzeros[p].
		\end{aligned}
	\end{equation}
	where $\map{f}{\real^n \times \real^m}{\real}$ and $\map{g}{\real^n \times \real^m}{\real^p}$ are each twice continuously differentiable on $\real^n \times \real^m$. Further assume that for some $x_0 \in \real^n$, $f(x_0, \cdot)$ is strongly convex and $g(x_0, \cdot)$ is convex and that there exists $\hat{u} \in \real^m$ such that $g(x_0,\hat{u}) \ll \vectorzeros[p]$\footnote{For two vectors, $v,w \in \real^n$, $v \ll w$ if $v_i < w_i$ for all $i \in \until{n}$.}. Then there exists a neighborhood of $x_0$ such that $u^\star$ is continuous at every point in the neighborhood.
	\begin{comment}
	\begin{enumerate}
		\item There exists a neighborhood of $x_0$ such that $u^\star$ is point-Lipschitz at every point in the neighborhood 
		\item $u^\star$ is H{\"o}lder at $x_0$ 
		\item $u^\star$ is directionally differentiable at $x_0$.
	\end{enumerate}
	\end{comment}
\end{prop}

By existence theorems, we know that for the system $\dot{x} = Ax + Bu^\star(x)$, for each initial condition $x_0$ satisfying the assumptions of Proposition~\ref{prop:regularity}, there exists a positive constant $\subscr{\tau}{max}(x_0)$ and a continuously differentiable curve $\map{\phi_{x_0}}{[0, \subscr{\tau}{max}(x_0))}{\real^n}$ satisfying
%\footnote{To be precise, for $\phi_{x_0}$ to be differentiable at $t = 0$, by Peano's existence theorem, there also exists $\subscr{\tau}{min}(x_0) < 0$ such that $\map{\phi_{x_0}}{(\subscr{\tau}{min}(x_0), \subscr{\tau}{max}(x_0))}{\real^n}$, but we will omit this parameter for simplicity.} 
$\frac{d\phi_{x_0}}{dt}(t) = A\phi_{x_0}(t) + Bu^\star(\phi_{x_0}(t))$ for all $t \in [0, \subscr{\tau}{max}(x_0))$. We say that the solution $\phi_{x_0}$ is forward-complete if $\subscr{\tau}{max}(x_0) = +\infty$.

\subsection{Stability Analysis for LTI Systems}
Consider the dynamical system and its corresponding controller defined via a parametric optimization problem~\eqref{eq:LTI} and define the following sets 
\begin{align}
	\Gamma(x) &:= \setdef{u \in \real^m}{g(x,u) \leq \vectorzeros[p]} \quad \text{ and } \\
	\mcK &:= \setdef{x \in \real^n}{\exists \hat{u} \text{ s.t. } g(x,\hat{u}) \ll \vectorzeros[p]},
\end{align}
where $\Gamma(x)$ represents the feasible control actions at the state $x$ and $\mcK$ denotes the points in state space where the feasible set, $\Gamma(x)$, has an interior.
%, and $\mcX:= \setdef{x_0 \in \real^n}{\phi_{x_0}(t) \in \mcK \text{ for all } t}$.

%% \begin{assumption} %\label{assm:big-assm}

  We make the following assumptions on our problem:
	\begin{enumerate}[label=\textup{(A\arabic*)},leftmargin=*]
		%\item\label{assm:1} The pair $(A,B)$ is stabilizable,
		\item\label{assm:1} \textbf{(Regularity of $g$)} The map $\map{g}{\real^n \times \real^m}{\real^p}$ is twice continuously differentiable on $\real^n \times \real^m$ and $g(x, \cdot)$ is convex for all $x \in \real^n$,
		\item\label{assm:1.5} \textbf{(Existence of equilibrium and feasibility of zero control)} $\vectorzeros[n] \in \mcK$ and $\vectorzeros[m] \in \Gamma(x)$ for all $x \in \mcK$,
		\item\label{assm:2} \textbf{(Linearity of nominal controller)} the map $\map{k}{\real^n}{\real^m}$ is linear, i.e., $k(x) = Kx$ for some $K \in \real^{m \times n}$,
		\item\label{assm:3} \textbf{(Dynamical feasibility)} for every $x_0 \in \mcK$, $\phi_{x_0}(t) \in \mcK$ for all $t \in [0,\taumax{x_0})$.
	\end{enumerate}

We make comments about some of these assumptions. 
Assumption~\ref{assm:1.5} ensures that $\vectorzeros[n]$ is an equilibrium point and that $u = \vectorzeros[m]$ is a feasible control action. Assumption~\ref{assm:3} ensures that the controller $u^\star$ does not drive the system outside the set of points where the feasible set of~\eqref{eq:LTI} has an interior. Outside of this set, the controller may fail to be continuous and solutions of~\eqref{eq:LTI} may fail to exist. One simple way to verify Assumption~\ref{assm:3} is to ensure that $\mcK = \real^n$. Note further that $u^\star(x)$ can compactly be written $u^\star(x) = \proj_{\Gamma(x)}(Kx)$, where given a nonempty, closed, convex set $\Omega \subseteq \real^m$, $\proj_{\Omega}(z) := \argmin_{v \in \Omega} \|z - v\|_2$. 

%\begin{remark}
%	In practice, one may define $u^\star(x) = Kx$ when $\Gamma(x) = \emptyset$, but this choice of controller is discontinuous and to establish stability, one needs to rely on solutions in the sense of Filippov. We leave the analysis of these discontinuous dynamics as future work.
%\end{remark}

We are now ready to state our first main theorem establishing the exponential stability of the system~\eqref{eq:LTI}.

\begin{comment}
Now let $x_0\in \real^n$ be arbitrary and let $\phi_{x_0}(t)$ denote the flow map of the interconnection~\eqref{eq:LTI} from initial condition $x(0) = x_0$ and define the virtual system 
\begin{equation}\label{eq:virtual-sys}
	\dot{y} = Ay + B\proj_{\Gamma(\phi_{x_0}(t))}(Ky).
\end{equation}
At fixed $t\geq 0$, this system is Lipschitz and has Jacobian defined almost everywhere.

Note that the dynamics~\eqref{eq:virtual-sys} is exactly a generalized Lur'e system since $\proj_{\Gamma(\phi_{x_0}(t))}$ is a time-varying nonlinearity satisfying some generalized slope bounds, i.e., $0 \preceq \jac{\proj_{\Gamma(\phi_{x_0}(t))}}(y) \preceq I_n$ for all $y$ for which the Jacobian exists. 

Also note that since $\proj$ is only Lipschitz, Proposition~\ref{prop:Lure1} does not immediately apply, but a trivial modification holds since the projection is almost everywhere differentiable. 
\end{comment}
\begin{theorem}[Exponential stability for LTI systems with parametric
    projection-based controllers]
  \label{thm:general-LTI}
  Consider the dynamics~\eqref{eq:LTI} satisfying
  Assumptions~\ref{assm:1}-\ref{assm:3}. Further suppose that there exist
  $P = P^\top \succ 0$, $\eta > 0$, and $\lambda \geq 0$ satisfying the inequality
  \begin{equation}\label{eq:LMI}
    \begin{bmatrix}
      A^\top P + PA + 2\eta P & PB + \lambda K^\top \\ B^\top P + \lambda K & -2\lambda I_m
    \end{bmatrix} \preceq 0.
  \end{equation}
  Then from any $x_0 \in \mcK$,
  \begin{enumerate}
  	\item\label{item:LTI-1} solutions to~\eqref{eq:LTI}, $\phi_{x_0}$, are forward-complete,
  	\item\label{item:LTI-2} the origin is globally exponentially stable with bound
  	\begin{equation}\label{eq:exp-stability-LTI}
  		\|\phi_{x_0}(t)\|_P \leq \e^{-\eta t}\|x_0\|_P,
  	\end{equation}
  	\item\label{item:LTI-3} the mapping $\map{V}{\mcK}{\real_{\geq 0}}$
          given by $V(x) = x^\top Px$ is a global Lyapunov function for the
          dynamics~\eqref{eq:LTI}.
  \end{enumerate} 
\end{theorem}
\begin{proof}
	We apply the virtual system method. Let $x_0 \in \mcK$ be arbitrary and consider the virtual system 
	\begin{equation}\label{eq:virtual-sys}
		\dot{y} = Ay + B\proj_{\Gamma(\phi_{x_0}(t))}(Ky).
	\end{equation}
	Note that for all $t \in [0,\taumax{x_0})$, $\proj_{\Gamma(\phi_{x_0}(t))}$ obeys the inequality~\eqref{eq:monotonicity} with $\rho = 1$ due to cocoercivity of projections, see, e.g.,~\cite[Eq. (2)]{EKR-SB:16}. Therefore the virtual system is a Lur'e system of the form~\eqref{eq:Lure}. Theorem~\ref{prop:Lure1} implies\begin{arxiv}
	\footnote{To apply Theorem~\ref{prop:Lure1}, we need to establish continuity in $t$ for the virtual system~\eqref{eq:virtual-sys}. 	We prove this property in Lemma~\ref{lem:virtual-sys-regularity}.}
	\end{arxiv}	
	that this virtual system is contracting with respect to $\|\cdot\|_P$ with rate $\eta > 0$. In other words, any two trajectories $y_1(\cdot), y_2(\cdot)$ for the virtual system satisfy for all $t \in [0,\taumax{x_0})$,
	\begin{equation}
		\|y_1(t) - y_2(t)\|_P \leq \e^{-\eta t}\|y_1(0) - y_2(0)\|_P.
	\end{equation}
	First note that $\phi_{x_0}$ is a valid trajectory for the virtual system so we set $y_1(t) = \phi_{x_0}(t)$. Additionally note that $y_2(t) = \vectorzeros[n]$ is a valid trajectory for the virtual system since $\vectorzeros[m] \in \Gamma(x)$ for all $x \in \mcK$. Substituting these two trajectories implies~\eqref{eq:exp-stability-LTI} for $t \in [0,\taumax{x_0})$. We now establish that $\taumax{x_0} = +\infty$. To this end, note that the bound~\eqref{eq:exp-stability-LTI} implies that the trajectory $\phi_{x_0}$ remains in the compact set $\setdef{x \in \real^n}{\|x\|_P \leq \|x_0\|_P}$ for $t \in [0,\taumax{x_0})$ for which it is defined. Since this set is compact, the trajectory cannot escape in a finite amount of time, meaning that the trajectory is forward complete. This reasoning proves statements~\ref{item:LTI-1} and~\ref{item:LTI-2}. Statement~\ref{item:LTI-3} is a consequence of~\ref{item:LTI-2}.
	\begin{arxiv}
	To prove statement~\ref{item:LTI-3} note that $V(x) = \|x\|_P^2$, let $x_0 \in \mcK$ and for $h > 0$ note that the inequality~\eqref{eq:exp-stability-LTI} implies
	\begin{equation*}
		\|\phi_{x_0}(h)\|_P^2 \leq \e^{-2\eta h}\|x_0\|_P^2.
	\end{equation*} 
	Now subtract $\|x_0\|_P^2$ from both sides, divide by $h > 0$ and take the limit as $h \to 0^+$ to conclude
	\begin{equation*}
		\lim_{h \to 0^+} \frac{\|\phi_{x_0}(h)\|_P^2 - \|x_0\|_P^2}{h} \leq \|x_0\|_{P}^2 \lim_{h \to 0^+} \frac{\e^{-2\eta h} - 1}{h}.
	\end{equation*}
	We see that the left-hand side of the above inequality is the definition of the Lie derivative of $V$ along trajectories of~\eqref{eq:LTI} with $x_0 \in \mcK$. Moreover, the right-hand side evaluates to $-2\eta\|x_0\|_P^2$. In other words, $\dot{V}(x_0) \leq -2\eta V(x_0)$ for all $x_0 \in \mcK$, which implies statement~\ref{item:LTI-3}.
	\end{arxiv}
\end{proof}

%\begin{remark}
%	It is clear that Hurwitzness of $A$ is a necessary condition for global exponential stability given the stated assumptions. Otherwise, $g$ may capture constraints such as $\|u\|_\infty \leq \bar{u}$, where $\bar{u}$ is some maximum control effort. With bounded control effort, one cannot expect to globally stabilize the system~\eqref{eq:LTI}. One possible remedy to this requirement is to pick $u = \tilde{K}x + v$ to stabilize the open-loop system and then use $v$ as the new control.
%\end{remark}

%The proof of Theorem~\ref{thm:general-LTI} highlights the utility of the virtual system method to establish global exponential stability. The virtual system method transforms a state-dependent Lur'e dynamics, $\dot{x} = Ax + B\proj_{\Gamma(x)}(Kx)$, into a time-varying one where Proposition~\ref{prop:Lure1} applies. 
The key insight in Theorem~\ref{thm:general-LTI} is that LTI systems with parametric projection controllers are a type of state-dependent Lur'e system $\dot{x} = Ax + B\proj_{\Gamma(x)}(Kx)$ and that the virtual system method transforms them into standard time-varying ones, where Theorem~\ref{prop:Lure1} applies. We remark that one could also use classical results such as the circle criterion to establish stability of~\eqref{eq:LTI} but we elect to use the virtual system method to highlight its utility.

In the case that Assumptions~\ref{assm:1}-\ref{assm:3}, $K = 0$, and the inequality~\eqref{eq:LMI} holds, the dynamics simply become $\dot{x} = Ax$ since $\vectorzeros[n] \in \Gamma(x)$ for all $x \in \mcK$. In other words, $K = 0$ always ensures global exponential stability under the stated assumptions since~\eqref{eq:LMI} implies that $A$ is Hurwitz. Theorem~\ref{thm:general-LTI} provides a novel sufficient condition for the system~\eqref{eq:LTI} to be exponential stable with $K \neq 0$ which has been designed to make $A + BK$ stable, e.g., using LQR. 

\subsection{Stability Analysis for Single-Integrators}

A special class of LTI systems for which we can ensure different results is the single integrator $\dot{x} = u^\star(x)$, namely the system~\eqref{eq:LTI} with $A = 0$ and $B = I_n$. For this class of systems, the inequality~\eqref{eq:LMI} will never hold with $\eta > 0$ since $A$ is not Hurwitz. 
\begin{theorem}[Exponential stability for single integrators with parametric
    projection-based controllers]\label{thm:single-integrator} Consider the
  dynamics~\eqref{eq:LTI} with $A = 0$ and $B = I_n$ and suppose
  Assumptions~\ref{assm:1}-\ref{assm:3} hold. Further suppose $K = K^\top
  \preceq -\eta I_n, \eta > 0$. Then from any $x_0 \in \mcK$,
  \begin{enumerate}
  	\item\label{item:single-1} solutions to~\eqref{eq:LTI}, $\phi_{x_0}$, are forward-complete and
  	\item\label{item:single-2} solutions asymptotically converge to the set of equilibria, $\subscr{\mcX}{eq} := \setdef{x \in \real^n}{\proj_{\Gamma(x)}(Kx) = \vectorzeros[n]}$.
  	\setcounter{saveenum}{\value{enumi}}
  \end{enumerate}
   Moreover, under the additional assumption that $\vectorzeros[n] \in \mathrm{int}(\Gamma(\vectorzeros[n]))$, the following statement holds:
   \begin{enumerate}\setcounter{enumi}{\value{saveenum}}
   	\item\label{item:single-3} if $\phi_{x_0}(t) \to \vectorzeros[n]$ as $t \to \infty$, then there exists $M(x_0) > 0$ such that
   	\begin{equation}\label{eq:semi-global}
   		\|\phi_{x_0}(t)\|_2 \leq M(x_0) \e^{-\eta t}\|x_0\|_2.
   	\end{equation}
   \end{enumerate}
\end{theorem}
\begin{proof}
	Consider as a Lyapunov function candidate $V(x) = -\frac{1}{2}x^\top Kx$. The Lie derivative of $V$ along trajectories of the dynamical system~\eqref{eq:LTI} is
	\begin{align*}
		\dot{V}(x) &= -x^\top K \proj_{\Gamma(x)}(Kx) \\
		&= -(Kx - \vectorzeros[n])^\top(\proj_{\Gamma(x)}(Kx) - \proj_{\Gamma(x)}(\vectorzeros))\\
		&\leq -\|\proj_{\Gamma(x)}(Kx) - \proj_{\Gamma(x)}(\vectorzeros[n])\|_2^2 \leq 0,
	\end{align*}
	where we have used that $\vectorzeros[n] \in \Gamma(x)$ for all $x$ and cocoercivity of $\proj_{\Gamma(x)}$, see, e.g.,~\cite[Eq.~(2)]{EKR-SB:16}. Since $\dot{V}(x) \leq 0$, we preclude finite escape time and thus conclude statement~\ref{item:single-1}. To establish asymptotic convergence to $\subscr{\mcX}{eq}$, we invoke LaSalle's invariance principle and see that trajectories converge to the largest forward-invariant set in $\setdef{x \in \real^n}{\dot{V}(x) = 0}$. However, since $\dot{V}(x) \leq -\|\proj_{\Gamma(x)}(Kx)\|_2^2$, $\dot{V}(x) = 0$ if and only if $\proj_{\Gamma(x)}(Kx) = \vectorzeros[n]$, i.e., $x \in \subscr{\mcX}{eq}$. This argument establishes statement~\ref{item:single-2}.
	%We adopt the virtual system approach. Let $x_0 \in \mcC$ be arbitrary and consider the virtual system
	%\begin{equation}\label{eq:virtual-sys-single-integrator}
	%	\dot{y} = \proj_{\Gamma(\phi_{x_0}(t))}(Ky).
	%\end{equation}
	%We compute the log norm of the Jacobian of the virtual system~\eqref{eq:virtual-sys-single-integrator} everywhere it exists and will establish that it is nonnegative everywhere. To this end, for any norm, let $\mu$ be its corresponding log norm and note that
	%\begin{align*}
	%	\sup_{y, t}\mu(\jac{\proj_{\Gamma(\phi_{x_0}(t))}}(Ky) K) \leq \max_{0 \preceq G \preceq I_n} \mu(GK),
	%\end{align*}
	%where we have leveraged that the projection is firmly nonexpansive. Now $G$ and $K$ are each symmetric, $G \succeq 0$ and $K \preceq 0$. As a consequence of Proposition~\ref{prop:sylvester}, picking the vector norm to be $\|\cdot\|_{-K}$, we have that $\max_{0 \preceq G \preceq I_n} \mu_{-K}(GK) \leq 0$. In other words, the virtual system is globally weakly contracting with respect to the norm $\|\cdot\|_{-K}$.
	%We can see global weak contraction of the virtual system~\eqref{eq:virtual-sys-single-integrator} by invoking Proposition~\ref{prop:Lure1} and seeing that the LMI~\eqref{eq:LMI} holds with $\eta = 0$ and $P = -K$.
	To establish statement~\ref{item:single-3}, note that $\vectorzeros[n] \in \mathrm{int}(\Gamma(\vectorzeros[n]))$ implies, by continuity of $g$, that there exist an open neighborhood, $\mcO_x$ containing the origin such that $g(x,Kx) \ll \vectorzeros[p]$ for all $x \in \mcO_x$. In other words, inside this neighborhood, $u^\star(x) = Kx$. Thus, the dynamics~\eqref{eq:LTI} are locally exponentially stable inside this neighborhood.
	Since the trajectory is asymptotically converging to the origin and locally exponentially stable inside $\mcO_x,$ we conclude~\eqref{eq:semi-global}.
	%note that the continuous and nonincreasing function $\map{\psi}{\realnonnegative}{\realnonnegative}$ defined by $\psi(t) = \|\phi_{x_0}(t)\|_{-K}$ satisfies $\lim_{t \to \infty} \psi(t) = 0$. Moreover, the dynamics~\eqref{eq:LTI} are locally exponentially stable inside $\mcO$ since the dynamics simply read $\dot{x} = -Kx$ inside $\mcO$. By asymptotic stability, there exists $T_c \geq 0$ such that $\phi_{x_0}(t) \in \mcO$ for all $t \geq T_c$ and by local exponential stability, we know that there exists $M_0 > 0$ such that for all $t \geq T_c$, $\psi(t) \leq M_0 \e^{-\eta t}\|x_0\|_{-K}$. 
	%Now we can see that the origin is an equilibrium point of the virtual system~\eqref{eq:LTI}. Inside the open set $\mcO$, for all $t$, the projection is not active and the virtual system dynamics simply read $\dot{y} = Ky$ and thus is locally contracting with respect to $\|\cdot\|_2$ with rate $\eta$. Since the virtual system~\eqref{eq:virtual-sys-single-integrator} is globally weakly contracting and locally contracting, we conclude semi-global exponential stability of the virtual system, i.e., there exists $M \geq 1$ such that for any trajectory $y(t)$ starting inside the compact set $\mcC$, the bound
	%\begin{equation}
	%	\|y(t)\|_2 \leq M\e^{-\eta t}\|y(0)\|
	%\end{equation}
	%holds. Since $\phi_{x_0}(t)$ is a trajectory of the virtual system starting inside the compact set $\mcC$, the result holds. Forward completeness of $\phi_{x_0}$ follows the same reasoning as in the proof of Theorem~\ref{thm:general-LTI}.
\end{proof}

To prove Theorem~\ref{thm:single-integrator}, one could alternatively use
the virtual system method to establish that $\|\phi_{x_0}(t)\|_{-K} \leq
\|x_0\|_{-K}$ and then invoke LaSalle's invariance principle. We opt to
give a more direct Lyapunov proof for simplicity.

\section{Applications}
\subsection{State-Dependent Saturation Control}
\begin{figure}[!t]
	\centering
	\includegraphics[width=0.95\linewidth]{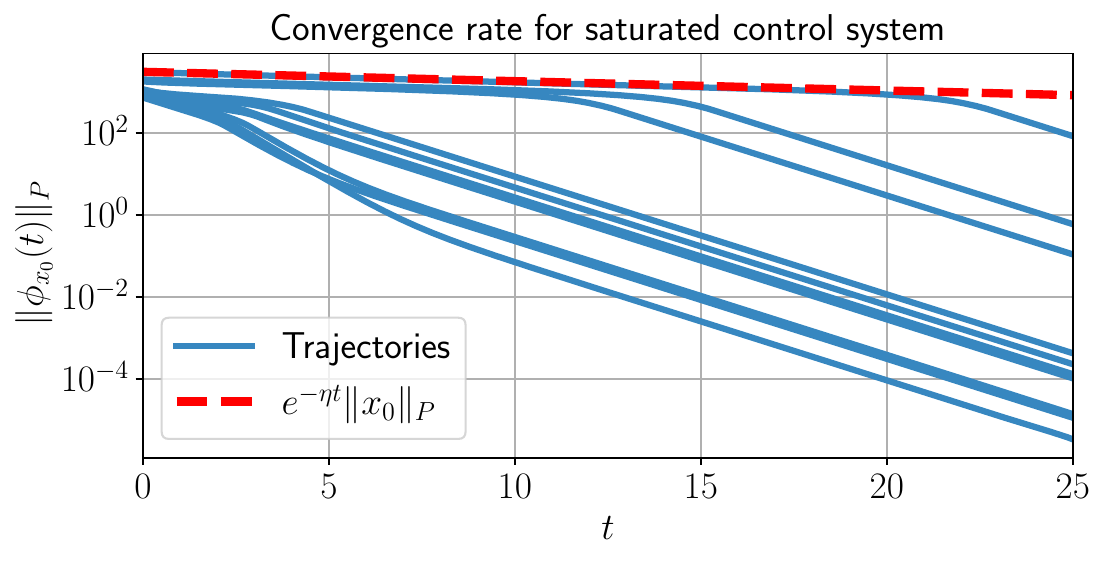} 
	\caption{The evolution of $\|\phi_{x_0}(t)\|_P$ for $10$ different $x_0$ and $P$ is chosen to maximize $\eta$ in~\eqref{eq:LMI}. We also plot $\e^{-\eta t}\|x_0\|_P$, where $\|x_0\|_P$ denotes the largest value of $\|x_0\|_P$ over all randomly generated initial conditions.}\label{fig:saturated}
\end{figure}
\begin{figure*}[!t]
	\centering
	\begin{tabular}{ccc}
		\includegraphics[height=5.57cm]{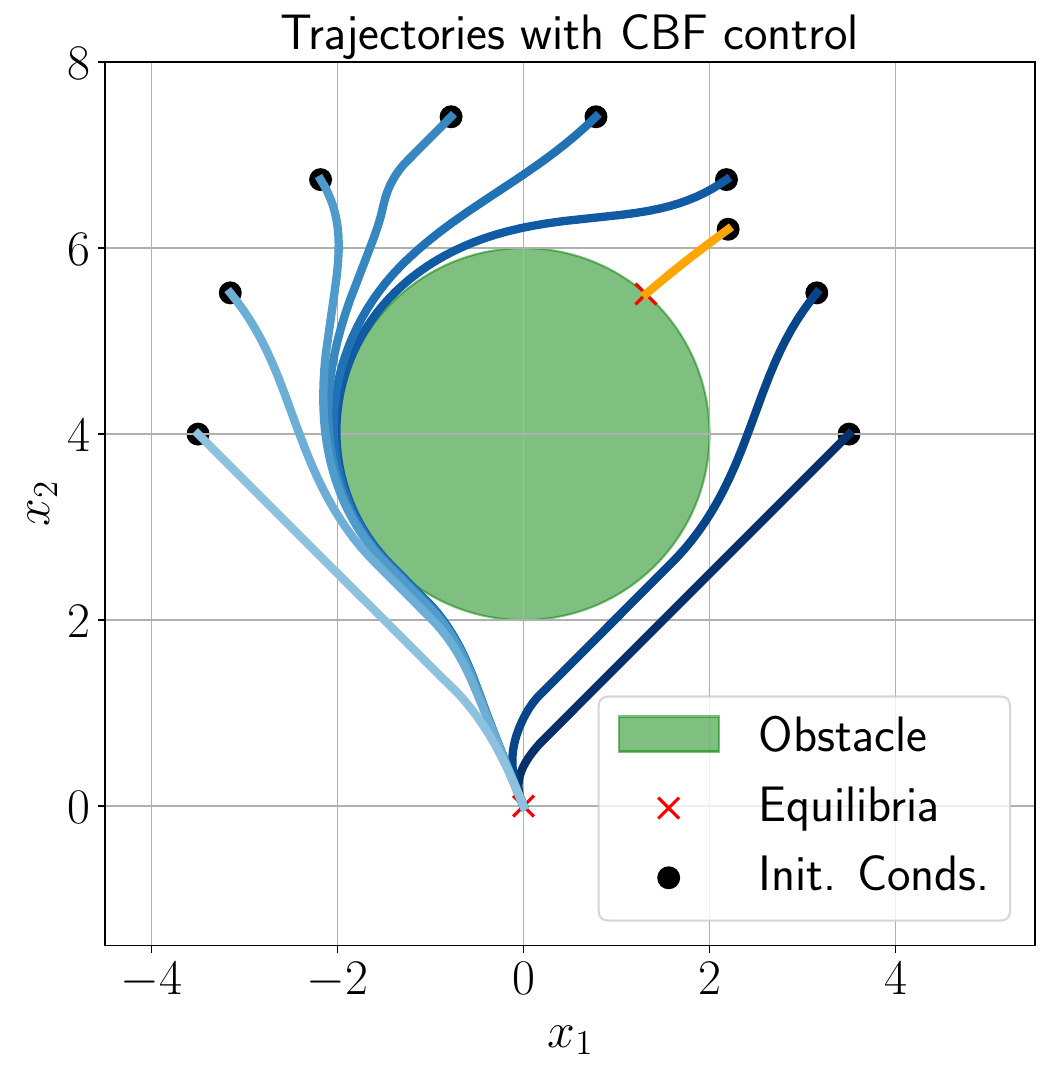} &
		\includegraphics[height=5.57cm]{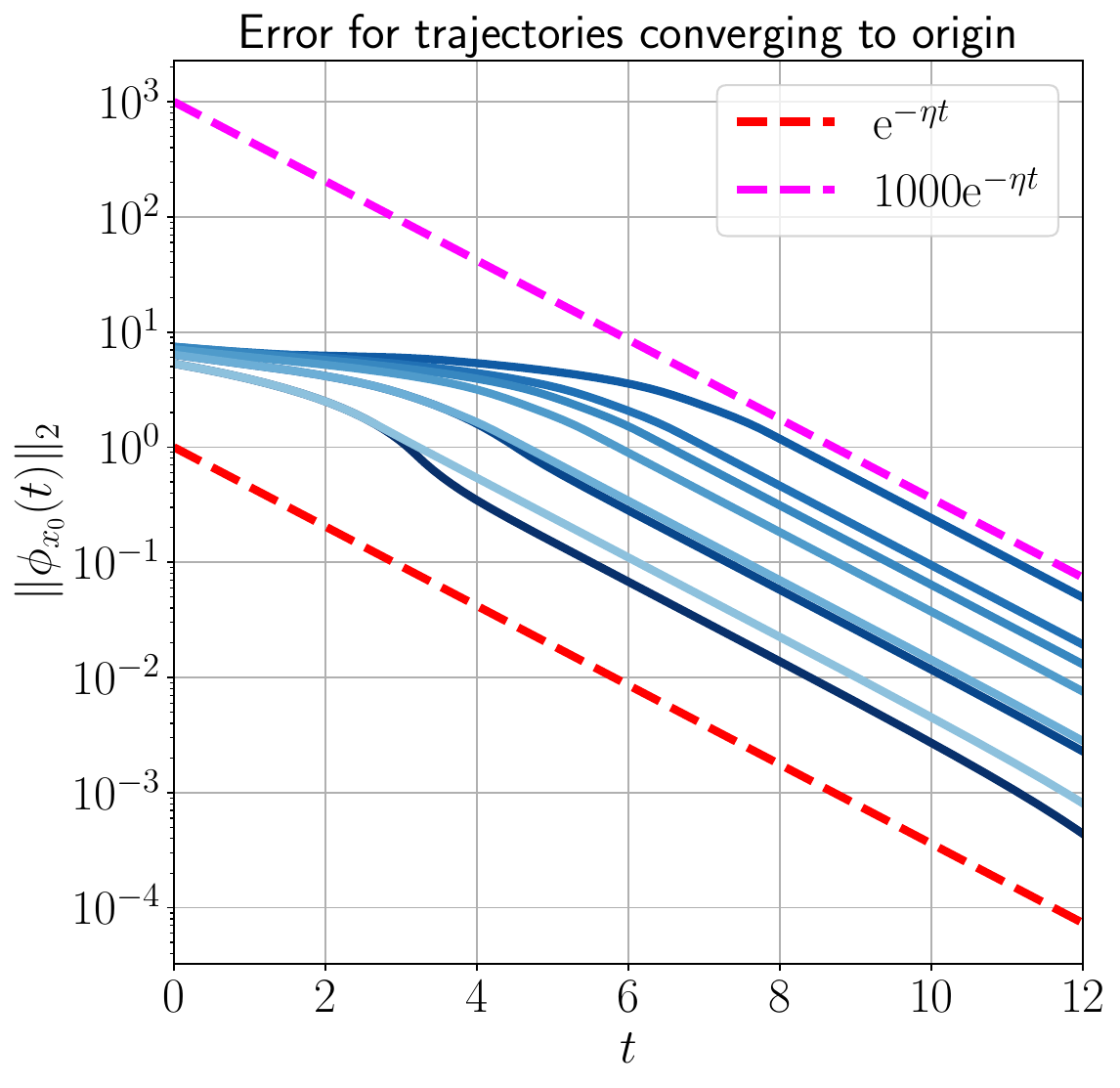} &
		\includegraphics[height=5.57cm]{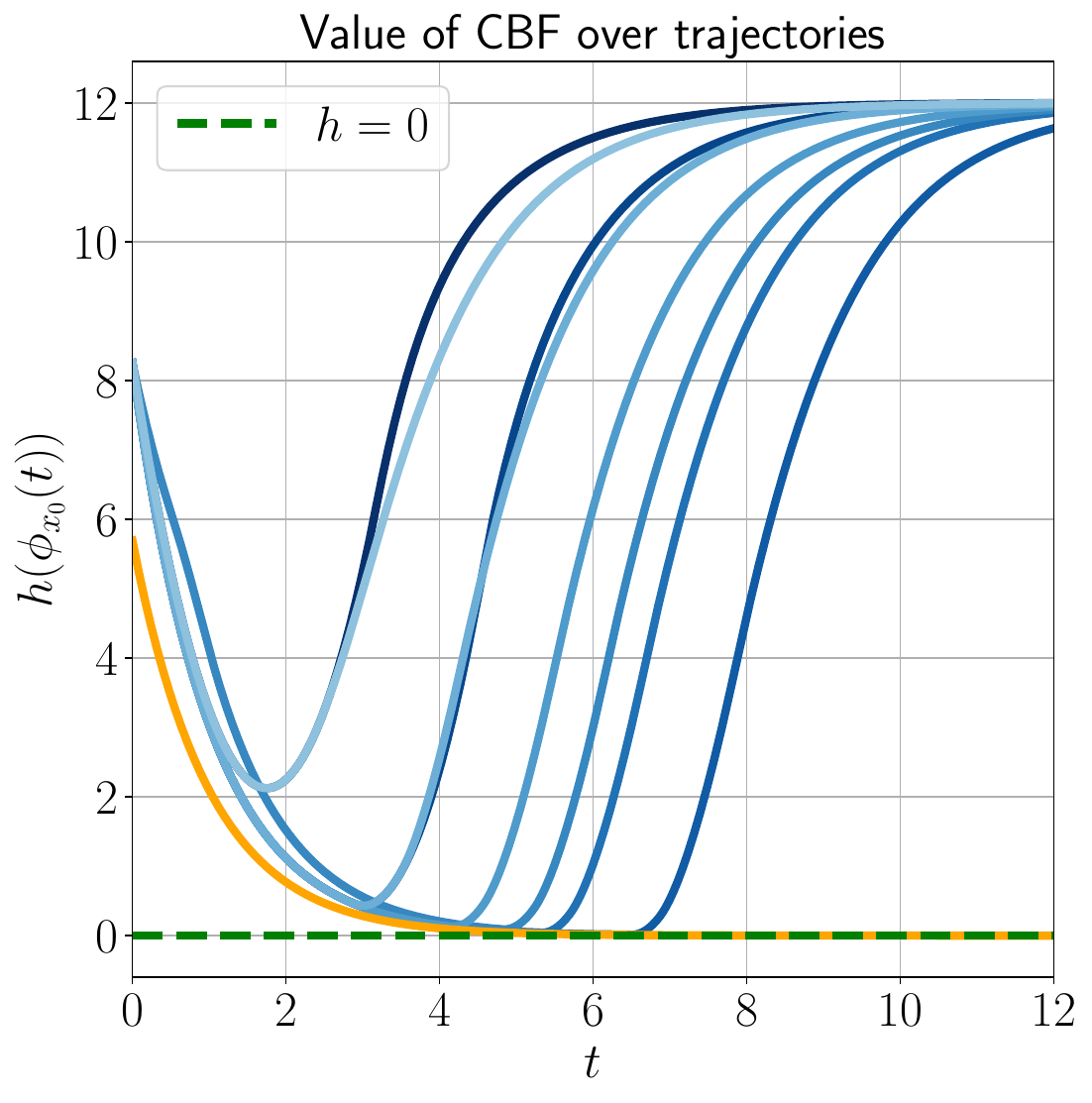}
	\end{tabular}
	\caption{The left figure shows plots of trajectories of~\eqref{eq:CBF} from various initial conditions. We can see that most trajectories, indicated by shades of blue, converge to the origin, while one converges to a point on the boundary of the safe set, shown in orange. The center figure plots the convergence rate of trajectories that converge to the origin. It also plots $\e^{-\eta t}$ and $1000e^{-\eta t}$ and demonstrates that the exponential convergence rate in~\eqref{eq:single-int-exp-conv} cannot be improved in this instance and that $1000 > M(x_0)\|x_0\|_2$ for these initial conditions. The right figure plots the evolution of the barrier function $h$ along trajectories and demonstrates how $h(\phi_{x_0}(t))$ remains positive for all $t \geq 0$. Best seen in color. }\label{fig:cbf-trajectories}
\end{figure*}
For $v \in \realpositive^n$, define the saturation function $\map{\sat_v}{\real^n}{\real^n}$ by
$\sat_v(x) := \max\{-v,\min\{v,x\}\},$
where the $\max$ and $\min$ are applied entrywise. An alternative characterization of the saturation function is via the minimization problem 
\begin{equation}
	\sat_v(x) = \argmin_{u \in \real^n} \setdef{\|u - x\|_2^2}{{-}v \leq u \leq v}.
\end{equation} 
We consider the state-dependent saturated control system
\begin{equation}\label{eq:state-dependent-sat}
	\dot{x} = Ax + B\sat_{v(x)}(Kx),
\end{equation}
where $\map{v}{\real^n}{\realpositive^n}$ is a twice continuously differentiable map dictating actuation constraints as a function of the state. 
%The control problem~\eqref{eq:state-dependent-sat} may be physically motivated if the state space is partitioned into different environments which affect the maximal control input. 
%For example, one could imagine that the dynamics~\eqref{eq:state-dependent-sat} models a wheeled robot which is driving in different terrains (e.g., sandy, rocky, smooth) which affect the maximum speed of the robot. 

When $v$ is constant, one can use results from saturated control systems to assess the stability of~\eqref{eq:state-dependent-sat}. As $v$ is state-dependent, one cannot apply these techniques here. It is straightforward to see that the system~\eqref{eq:state-dependent-sat} is of the form~\eqref{eq:LTI} with $g(x,u) = (u - v(x), -u-v(x))$. Moreover, it is routine to establish that Assumptions~\ref{assm:1}-\ref{assm:3} hold. Therefore, Theorem~\ref{thm:general-LTI} may be applied to provide a sufficient condition for the global exponential stability of the system~\eqref{eq:state-dependent-sat}.
\begin{example}
	We consider the system~\eqref{eq:state-dependent-sat}, where $n = 3, m = 2$, $A =-I_3 + N$, $B = \left[\begin{smallmatrix} 1 & 0 & 0 \\ 0 & 1 & 0 \end{smallmatrix}\right]^\top$, and $N \in \real^{3 \times 3}$ is a random matrix with entries drawn from the standard normal distribution. We assume that $K \in \real^{2 \times 3}$ is preselected so that $u = Kx$ minimizes the LQR objective $\int_0^\infty (x(t)^\top x(t) + u(t)^\top u(t)) dt$ for the LTI system $\dot{x} = Ax + Bu$. We also let $v(x) = \e^{-\|x\|_2^2/2}\vectorones[m]$, where $\vectorones[m] \in \real^m$ is the all-ones vector. We find $\eta > 0, \lambda \geq 0, P \in \real^{3 \times 3}$ satisfying~\eqref{eq:LMI} such that $\eta$ is maximized and plot values of $\|\phi_{x_0}(t)\|_{P}$ for $10$ different samples of $x_0$ from the multivariate normal distribution $\mathcal{N}(\vectorzeros[3],4 I_3)$ in Figure~\ref{fig:saturated}.
	
	We can see that all trajectories converge exponentially quickly to the origin and that the estimated exponential convergence rate from Theorem~\ref{thm:general-LTI} is $\eta = 0.0525$. Empirically, we can see that when trajectories are far from the origin, this rate is tight since $v(x) \approx \vectorzeros[m]$ for $x$ far from $\vectorzeros[3]$ and that the rate is very loose when trajectories are close to the origin since $v(x) \approx \vectorones[m]$ for $x \approx \vectorzeros[3]$.
\end{example}
\subsection{Stability with Control Barrier Functions}
Consider the nonlinear control-affine system
\begin{equation}\label{eq:control-affine}
\dot{x} = F(x) + G(x)u,
\end{equation}
where $\map{F}{\real^n}{\real^n}, \map{G}{\real^n}{\real^{n \times m}}$ are locally Lipschitz. 

Let $\mcC \subseteq \real^n$ and $\map{h}{\real^n}{\real}$ be a sufficiently smooth function such that $\mcC = \setdef{x \in \real^n}{h(x) \geq 0}$. The set $\mcC$ is referred to as the ``safe set". 
\begin{definition}[{Control Barrier Function~\cite[Definition~3]{ADA-XX-JWG-PT:17}}]
	The function $h$ is a control barrier function (CBF) for $\mcC$ if there exists a locally Lipschitz and strictly increasing function $\map{\alpha}{\real}{\real}$ with $\alpha(0) = 0$ such that for all $x \in \mcC$, there exists $u \in \real^m$ with
	\begin{equation}\label{eq:CBF-condition}
		\nabla h(x)^\top F(x) + \nabla h(x)^\top G(x) u + \alpha(h(x)) \geq 0.
	\end{equation}
\end{definition}

A continuous controller $\map{u}{\real^n}{\real^m}$ which satisfies~\eqref{eq:CBF-condition} for all $x \in \mcC$ renders $\mcC$ forward-invariant under the dynamics~\eqref{eq:control-affine}~\cite[Theorem~4]{RK-ADA-SC:21}.

A common way to synthesize controllers that render $\mcC$ forward invariant is via a parametric QP~\cite{ADA-XX-JWG-PT:17}. To this end, we consider a single-integrator being driven by the CBF constraint~\eqref{eq:CBF-condition} and actuator constraints:

\begin{equation}\label{eq:CBF}
	\begin{aligned}
		\dot{x} &= u^\star(x), \\
		u^{\star}(x) := &\argmin_{u \in \real^m} &&\frac{1}{2}\|u - Kx\|_2^2 \\
		& \text{s.t.} &&{-}\nabla h(x)^\top u \leq \alpha(h(x)),\\
		&  &&{-}\bar{u}\vectorones[n] \leq u \leq \bar{u}\vectorones[n],
	\end{aligned}
\end{equation}
where $\bar{u}> 0$.
To study convergence of~\eqref{eq:CBF}, we can check for conditions under which the hypotheses of Theorem~\ref{thm:single-integrator} hold. 
%As before, define the point-to-set maps $\Gamma(x) := \setdef{u \in \real^m}{{-}\nabla h(x)^\top Bu \leq \nabla h(x)^\top Ax + \alpha(h(x))}$. Note that the relative degree one condition implies that for all $x \in \interior(\mcC)$, $\Gamma(x)$ is nonempty. Moreover, by control barrier function theory, we have that the safe set $\mcC$ is forward invariant. 

\begin{corollary}[Exponential stability for single integrators with CBF-based controllers]
  Consider the dynamics~\eqref{eq:CBF} and suppose that (i) $h$ is a CBF
  for $\mcC$, (ii) $K = K^\top \preceq -\eta I_n$, (iii) $\vectorzeros[n]
  \in \mathrm{int}(\mcC)$, and (iv) $h$ is thrice continuously
  differentiable and $\alpha$ is twice continuously differentiable. Then
  from any $x_0 \in \mathrm{int}(\mcC)$,
  \begin{enumerate}
  \item solutions to~\eqref{eq:CBF}, $\phi_{x_0}$, are forward-complete, 
  \item solutions remain in $\mcC$ for all $t \geq 0$,
  \item solutions converge to the set of equilibria
    %, $\subscr{\mcX}{eq} := \setdef{x \in \real^n}{\proj_{\Gamma(x)}(Kx) = \vectorzeros[n]}$, where $\Gamma(x) := \setdef{u \in \real^m}{{-}\nabla h(x)^\top u \leq \alpha(h(x)), \text{ and } \|u\|_\infty \leq \bar{u}}$, 
		and
		\item if $\phi_{x_0}(t) \to \vectorzeros[n]$ as $t \to \infty$, then there exists $M(x_0) > 0$ such that
		\begin{equation}\label{eq:single-int-exp-conv}
			\|\phi_{x_0}(t)\|_2 \leq M(x_0)\e^{-\eta t}\|x_0\|_2.
		\end{equation}
	\end{enumerate}
\end{corollary}
\begin{proof}
	\begin{lcss}
		It is straightforward to verify Assumptions~\ref{assm:1}-\ref{assm:3} and that $\vectorzeros[n] \in \interior(\Gamma(\vectorzeros[n]))$.
	\end{lcss}
	\begin{arxiv}
	We simply need to verify the assumptions of Theorem~\ref{thm:single-integrator}, namely Assumptions~\ref{assm:1}-\ref{assm:3}, and that $\vectorzeros[n] \in \mathrm{int}(\Gamma(\vectorzeros[n]))$. We can see Assumption~\ref{assm:1} holds by smoothness of $h$ and $\alpha$. We can see that Assumption~\ref{assm:1.5} holds since $\vectorzeros[n] \in \mathrm{int}(\mcC)$ and that $\vectorzeros[n] \in \Gamma(x)$ for $x \in \mcK$ by the assumption that $h$ is a CBF for $\mcC$ and that $\alpha(h(x)) > 0$ for $x \in \mathrm{int}(\mcC)$. Assumption~\ref{assm:2} is clear and Assumption~\ref{assm:3} is guaranteed since the feasible set of~\eqref{eq:CBF} has an interior for all $x \in \mathrm{int}(\mcC)$ since $\alpha(h(x)) > 0$ for $x \in \mathrm{int}(\mcC)$. Finally, we can see that $\vectorzeros[n] \in \mathrm{int}(\Gamma(\vectorzeros[n]))$ since $\vectorzeros[n] \in \mathrm{int}(\mcC)$. Since these assumptions hold, the consequences of Theorem~\ref{thm:single-integrator} apply.
	\end{arxiv}
\end{proof}
In general, we cannot conclude uniqueness of equilibria. As we will see in the following example, the dynamics may have multiple equilibria, even in the case of simple CBFs. These results agree with the theory presented in~\cite{MFR-APA-PT:21,XT-DVD:24}.

\begin{example}
	Consider a two-dimensional single integrator attempting to avoid a disk-shaped obstacle centered at $(0,4)$ with radius $2$. The corresponding CBF is $h(x_1,x_2) = x_1^2 + (x_2-4)^2 - 4$ with $\alpha(r) = r$. We take $K = \left[\begin{smallmatrix} -2 & -0.5 \\ -0.5 & -1 \end{smallmatrix}\right]$ and $\bar{u} = 1$. We plot numerical simulations of~\eqref{eq:CBF} with these parameters along with the corresponding convergence rate and value of CBF along trajectories in Figure~\ref{fig:cbf-trajectories}.
	
	We can observe that all trajectories converge to the set of equilibria and the majority of them converge to the origin. Moreover, we see that the exponential convergence rate predicted in~\eqref{eq:single-int-exp-conv} appears to be tight in this example and that all trajectories remain in $\mathrm{int}(\mcC)$ for all $t \geq 0$. While the equilibrium point on the boundary of the safe set is unavoidable in this example, we can numerically observe that there are no equilibria inside $\interior(\mcC)$ other than the origin. This result is in contrast with the CLF-CBF controllers studied in~\cite{MFR-APA-PT:21,XT-DVD:24}, whereby there may exist additional undesirable equilibria inside $\interior(\mcC)$. This numerical example provides evidence that one does not need to enforce any CLF decrease condition to ensure convergence to the origin, except on a set of measure zero.
\end{example}

\section{Discussion and Future Work}
In this letter we study LTI systems with controllers solving a special class of parametric programs, namely parametric projections. Using the virtual system method and a novel contractivity result for Lur'e systems, we provide a novel set of sufficient conditions for the exponential stability of these systems. For the special case of single integrators, we prove convergence to the set of equilibria and an exponential convergence rate for trajectories converging to the origin. As applications, we apply our results to state-dependent saturated control systems and to CBF-based control. 

We believe that there are many avenues to extend the results in this
letter. First, it would be interesting to see how the
Assumptions~\ref{assm:1}-\ref{assm:3} can be relaxed to allow for a larger
class of LTI systems, possibly leveraging the analysis tools
in~\cite{AVP-AD-FB:22a}. Second, we believe that our analysis method could
continue to apply for classes of structured nonlinear systems such as
controlled Lur'e systems. Finally, it is important to understand properties
of systems whose controller only approximately solve the parametric
program.

\appendix
\newcommand{\mcN}{\mathcal{N}}
\newcommand{\mcV}{\mathcal{V}}
\newcommand{\dist}{\mathrm{dist}}
\begin{arxiv}
	\subsection{Auxiliary Results in Set-Valued Analysis}
	%Standard notions from set-valued analysis are available in~\cite[Section~2.1.3]{FF-JSP:03} (i.e., domain, continuity, limits).
	We recall some definitions from set-valued analysis and refer to~\cite[Section~2.1.3]{FF-JSP:03} for a more comprehensive treatment. 
	
	A point-to-set map, also called a set-valued map, is a map, $K$, from $\real^n$ to the power set of $\real^m$, i.e., the set of subsets of $\real^m$. We will denote a point-to-set map by $K: \real^n \rightrightarrows \real^m$. The \emph{domain} of $K$ is denoted $\dom(K)$ and is the set $\dom(K) := \setdef{x \in \real^n}{K(x) \neq \emptyset}$. 
	\begin{definition}[Lower semicontinuity]
		A set-valued map $K: \real^n \rightrightarrows \real^m$ is \emph{lower semicontinuous} at a point $\bar{x}$ if for every open set $\mcU \subseteq \real^m$ such that $K(\bar{x}) \cap \mcU \neq \emptyset$, there exists an open neighborhood $\mcN \subseteq \real^n$ of $\bar{x}$ such that, for each $x \in \mcN$, $K(x) \cap \mcU \neq \emptyset$.
	\end{definition}
	
	\begin{definition}[Upper semicontinuity]
		A set-valued map $K: \real^n \rightrightarrows \real^m$ is \emph{upper semicontinuous} at a point $\bar{x}$ if for every open set $\mcV \subseteq \real^m$ containing $K(\bar{x})$, there exists an open neighborhood $\mcN \subseteq \real^n$ of $\bar{x}$ such that, for each $x \in \mcN$, $\mcV$ contains $K(x)$.
	\end{definition}
	We say a set-valued map $K: \real^n \rightrightarrows \real^m$ is \emph{continuous} at $x$ if it is both lower semicontinuous and upper semicontinuous at $x$. 
	
	Suppose that $K: \real^n \rightrightarrows \real^m$ is a \emph{closed-valued} set-valued map, i.e., $K(x)$ is a closed set for all $x \in \dom(K)$. Then for every $x \in \dom(K)$, the distance function $\dist(\cdot,K(x))$ is well-defined, where
	\begin{equation}
	\dist(z,K(x)) = \inf_{y \in K(x)} \|y - z\|_2.
	\end{equation}
	The \emph{liminf} of $K(y)$ as $y$ tends to $x$, denoted $\liminf_{y \to x} K(y)$ is defined to be the set of vectors $z \in \real^m$ such that
	\begin{equation}
	\lim_{y \to x} \dist(z,K(y)) = 0.
	\end{equation}
	Similarly, the \emph{limsup} of $K(y)$ as $y$ tends to $x$, denoted by $\limsup_{y \to x} K(y)$ is defined to be the set of vectors $z \in \real^m$ such that
	\begin{equation}
	\liminf_{y \to x} \dist(z,K(y)) = 0.
	\end{equation}
	Note that for every $x \in \real^n$,
	$\liminf_{y \to x} K(y) \subseteq K(x) \subseteq \limsup_{y \to x} K(y)$. If equality holds between the liminf and limsup sets, we write the common set as $\lim_{y \to x} K(y)$ and call it the limit of $K(y)$ as $y$ tends to $x$.
	
	We will now require some auxiliary results.
	
	\begin{prop}[{\cite[Proposition~2.1.17(c)]{FF-JSP:03}}]\label{prop:continuous-limits}
		Suppose $K: \real^n \rightrightarrows \real^m$ is a closed-valued point-to-set map. 
		If $K$ is continuous at $x$, then
		\begin{equation}
		\lim_{y \to x} K(y) = K(x).
		\end{equation}
	\end{prop}
	
	\begin{prop}[{\cite[Lemma~2.8.2]{FF-JSP:03}}]\label{prop:continuous-parametric-proj}
		Let $K: \real^p \rightrightarrows \real^n$ be a closed-valued and convex-valued point-to-set map. Let $\bar{x} \in \dom(K)$ be given. The (single-valued) map $\Phi(x,y) := \proj_{K(x)}(y)$ is continuous at $(\bar{x},y)$ for all $y \in \real^n$ if and only if
		\begin{equation}
			\lim_{x \to \bar{x}} K(x) = K(\bar{x}).
		\end{equation}
	\end{prop}
	
	\begin{prop}[{\cite[Exercise 2.9.34]{FF-JSP:03}}]\label{prop:scalar-continuous}
		Let $\map{g}{\real^{n} \times \real^m}{\real}$ be continuous. Let $\bar{x} \in \real^n$, such that $g(\bar{x},\cdot)$ is a convex function and there exists $\bar{u} \in \real^m$ such that $g(\bar{x},\bar{u}) < 0$. Then the set-valued map $\Gamma: \real^n \rightrightarrows \real^m$ defined by 
		\begin{equation}
			\Gamma(x) := \setdef{u \in \real^m}{g(x,u) \leq 0},
		\end{equation}
		is continuous at $\bar{x}$.
	\end{prop}
	
	We now are equipped with all the background needed to establish continuity of the parametric projection.
	
	\begin{lemma}\label{cor:proj-continuity}
		Let $\map{g}{\real^n \times \real^m}{\real^p}$ be continuous. Let $\bar{x} \in \real^n$ such that $g(\bar{x},\cdot)$ is a convex function and there exists $\bar{u} \in \real^m$ such that $g(\bar{x},\bar{u}) \ll \vectorzeros[p]$. Define the set-valued map $\Gamma: \real^n \rightrightarrows \real^m$ given by
		\begin{equation}
			\Gamma(x) := \setdef{u \in \real^m}{g(x,u) \leq \vectorzeros[p]}.
		\end{equation}
		Then the map $\Phi(x,u) := \proj_{\Gamma(x)}(u)$ is continuous at $(\bar{x},u)$ for all $u \in \real^m$.
	\end{lemma}
	\begin{proof}
		We simply need to extend Proposition~\ref{prop:scalar-continuous} to the case that $g$ is vector-valued. For $i \in \until{p}$, define the set-valued maps $\Gamma_i(x) := \setdef{u \in \real^m}{g_i(x,u) \leq 0}$ and note that $\Gamma(x) = \bigcap_{i=1}^p \Gamma_i(x)$. We simply need to show upper and lower semicontinuity of $\Gamma$ to conclude continuity. Regarding lower-semicontinuity, let $\bar{x} \in \dom(\Gamma)$ be such that there exists $\bar{u}$ with $g(\bar{x},\bar{u}) \ll \vectorzeros[p]$ and let $\mcX \subseteq \real^n$ be open and satisfy $\Gamma(\bar{x}) \cap \mcX \neq \emptyset$. By continuity of each of the $\Gamma_i$ at $\bar{x}$ (due to Proposition~\ref{prop:scalar-continuous}), there exists an open set $\mathcal{N}_i$ containing $\bar{x}$ such that for each $x^i \in \mathcal{N}_i$, $\Gamma_i(x^i) \cap \mcX \neq \emptyset$ for all $i \in \until{p}$. Since a finite intersection of open sets is open, let $\mathcal{N} := \bigcap_{i=1}^p \mathcal{N}_i$. Then it is routine to check that for each $x \in \mathcal{N}$, $\Gamma(x) \cap \mcX \neq \emptyset$. This establishes lower semicontinuity. 
		
		To establish upper semicontinuity, let $\bar{x}$ satisfy the stated assumptions and let $\mathcal{V}$ be an open set with $\Gamma(\bar{x}) \subseteq \mathcal{V}$. By continuity of each of the $\Gamma_i$ at $\bar{x}$, there exists an open neighborhood $\mathcal{N}_i$ of $\bar{x}$ such that for each $x^i \in \mathcal{N}_i$, $\Gamma_i(x) \subseteq \mathcal{V}$. Once again taking $\mathcal{N} = \bigcap_{i=1}^p \mathcal{N}$ ensures that for all $x \in \mathcal{N}$, we have that $\Gamma(x) \subseteq \mathcal{V}$. This establishes upper semicontinuity and thus continuity of $\Gamma$ at $\bar{x}$. 
		
		The result is then a consequence of Propositions~\ref{prop:continuous-limits} and~\ref{prop:continuous-parametric-proj}.
	\end{proof}
\end{arxiv}

\begin{arxiv}
	\newcommand{\mcA}{\mathcal{A}}
	\subsection{A Continuity Result from Fixed Point Theory}
	To establish continuity and Lipschitzness of solutions to parametric optimization problems, it is oftentimes useful to first express them as fixed point problems. From there, we can leverage results from a parametrized version of the Banach fixed point theorem. We reproduce this result in the following proposition.
	\begin{prop}[{\cite[Section I.6.A. Item (A.4) and Section I.3. Theorem 3.2]{AG-JD:03}}]\label{lem:contr-continuity}
		Let $(\mcX,d)$ be a complete metric space, $(\mcA,\varrho)$ be a metric space, and $\map{f}{\mcX \times \mcA}{\mcX}$ satisfy the following properties:
		\begin{enumerate}
			\item There exists $\ell < 1$ such that for each $a \in \mcA$, the map $f(\cdot,a)$ is contractive with constant $\ell$, i.e., for all $x_1,x_2 \in \mcX$,
			\begin{equation}
				d(f(x_1,a),f(x_2,a)) \leq \ell d(x_1,x_2).
			\end{equation}
			
			\item For each $x \in \mcX$, the map $a \mapsto f(x,a)$ is continuous.
		\end{enumerate}
		Then for each $a \in \mcA$, there exists a unique $x^\star(a)$ satisfying the fixed point equation $x^\star(a) = f(x^\star(a),a)$ and the mapping $a \mapsto x^\star(a)$ is continuous.
		
		Moreover, if there exists $L > 0$ such that for each $x \in \real^n$, the map $a \mapsto f(x,a)$ is Lipschitz with constant $L$, then the mapping $a \mapsto x^\star(a)$ is Lipschitz with estimate $L/(1-\ell)$.
	\end{prop}
\end{arxiv}

\begin{arxiv}
	\subsection{Well-Posedness of Virtual System}

Given $x_0 \in \mcK$, we're interested in the properties of the map $\map{\bar{u}}{\real^m \times [0,\taumax{x_0})}{\real^m}$ given by
\begin{equation}\label{eq:virtual-opt}
	\begin{aligned}
		\bar{u}(z,t) := &\argmin_{u \in \real^m} &&f(z,u) \\
		& \text{s.t.} && g(\phi_{x_0}(t),u) \leq \vectorzeros[p].
	\end{aligned}
\end{equation}
This is exactly the map that appears in the virtual system analysis, i.e., for $f(z,u) = \frac{1}{2}\|u - z\|_2^2$, the virtual system in the proof of Theorem~\ref{thm:general-LTI} is $\dot{y} = Ay + B\bar{u}(Ky,t)$. 
%The following result establishes that this virtual system is Lipschitz in $y$ and continuous in $t$.

\begin{lemma}\label{lem:virtual-sys-regularity}
	Suppose that for all $z \in \real^n$, $f(z, \cdot)$ is strongly convex and $g(z, \cdot)$ is convex. Consider the ODE $\dot{x} = Ax + Bu^\star(x)$, where $u^\star(x)$ is the solution to~\eqref{eq:PNLP}. For $x_0 \in \mcK$, let $\phi_{x_0}$ denote a solution of the ODE from initial condition $x_0$ on the interval $[0,\taumax{x_0})$. Assume that for all $t \in [0,\taumax{x_0})$ that $\phi_{x_0}(t) \in \mcK$ and that $\nabla_u f$ is Lipschitz in both of its arguments. Then the map $\bar{u}$ in~\eqref{eq:virtual-opt} is Lipschitz in its first argument and continuous in its second.
\end{lemma}
\end{arxiv}

\begin{arxiv}
	\begin{proof}
		We first make an important observation. The value of $\bar{u}$ at the point $(z,t)$ solves the fixed point equation
		\begin{equation}
			\bar{u}(z,t) = \proj_{\Gamma(\phi_{x_0}(t))}(\bar{u}(z,t) - \gamma \nabla_u f(z,\bar{u}(z,t))),
		\end{equation}
		where $\gamma > 0$ is a constant to be tuned to make the map contractive. Let $\map{\bar{f}}{\real^m \times \real^n \times [0,\taumax{x_0})}{\real^m}$ denote this fixed-point map, i.e., $\bar{f}(u,z,t) := \proj_{\Gamma(\phi_{x_0}(t))}(u - \gamma \nabla_u f(z,u)).$ Under the assumption that $\nabla_u f$ is Lipschitz, from standard results on projected gradient algorithms, e.g.,~\cite[pp.~31]{EKR-SB:16}, we know that for fixed $z,t$, the map $u \mapsto \bar{f}(u,z,t)$ is a contraction for $0 < \gamma < 2/L$, where $L$ is the Lipschitz constant of $\nabla_u f(z,\cdot)$. Henceforth, we let $\gamma$ be in this range. To establish Lipschitzness in $z$, we note that at fixed $u,t$, the mapping $z \mapsto \bar{f}(u,z,t)$ is Lipschitz under the assumption that the mapping $\nabla_u f(\cdot,u)$ is Lipschitz. Thus, at fixed $t$, we can apply Proposition~\ref{lem:contr-continuity} to conclude that the mapping $z \mapsto \bar{u}(z,t)$ is Lipschitz and thus continuous in its first argument. 
		
		To establish continuity in $t$, we first study the point-to-set mapping $t \rightrightarrows \Gamma(\phi_{x_0}(t))$. By assumption that for all $t$, we have that $\phi_{x_0}(t) \in \mcK$, we conclude by Lemma~\ref{cor:proj-continuity} that the mapping $\map{\Phi}{\real^m \times [0,\taumax{x_0})}{\real^m}$ given by $\Phi(u,t) := \proj_{\Gamma(\phi_{x_0}(t))}(u)$ is continuous on its domain. Since $\bar{f}(u,z,t) = \Phi(u - \gamma \nabla_u f(z,u),t)$, we have that $\bar{f}(u,z,\cdot)$ is continuous as well. And since for all $z,t$, we know that $u \mapsto \bar{f}(u,z,t)$ is a contraction, by Proposition~\ref{lem:contr-continuity}, we know that for all $z \in \real^m$, the map $t \mapsto \bar{u}(z,t)$ is continuous.
	\end{proof}
\end{arxiv}
%So now we have established that flows of the nominal ODE $\dot{x} = Ax + Bu^\star(x)$ exist for at least some interval of time and that
%an ODE of the form $\dot{y}(t) = Ay(t) + B\bar{u}(Ky(t),t)$ is Lipschitz in its state, $y$, and continuous in $t$, so it will have unique solutions on the interval $[0,\taumax{x_0})$. In particular, we will be studying contraction-theoretic properties of this auxiliary $y$-ODE to infer stability guarantees for the original $x$-ODE.
\begin{lcss}
	\bibliographystyle{IEEEtran}
\end{lcss}
\begin{arxiv}
	\bibliographystyle{plainurl+isbn}
\end{arxiv}
\bibliography{alias,Main,FB,New}

\end{document}